\documentclass[12pt,reqno]{amsart}

\usepackage{amssymb,amsthm}
\usepackage{amsfonts,cmtiup,comment,stmaryrd}

\usepackage{mathrsfs,cmtiup}

\makeatletter
\def\blfootnote{\xdef\@thefnmark{}\@footnotetext}
\makeatother

\newtheorem{theorem}{Theorem}[section]

\newtheorem{lemma}[theorem]{Lemma}
\newtheorem{proposition}[theorem]{Proposition}
\newtheorem{corollary}[theorem]{Corollary}

\theoremstyle{definition}

\newtheorem*{definition*}{Definition}

\let\leq=\leqslant
\let\geq=\geqslant
\numberwithin{equation}{section}
\newcommand{\ed}{\end{document}}
\newcommand{\bp}{\begin{proof} }
\newcommand{\ep}{\end{proof} }
\newcommand{\N}{{\Bbb N}}

\begin{document}
\title{Compact groups\\ all elements of which are almost right Engel}

\author{E. I. Khukhro}
\address{Charlotte Scott Research Centre for Algebra,\newline\indent University of Lincoln, Lincoln, LN6 7TS, U.K., and\newline\indent
  Sobolev Institute of Mathematics, Novosibirsk, 630090, Russia}
\email{khukhro@yahoo.co.uk}

\author{P. Shumyatsky}

\address{Department of Mathematics, University of Brasilia, DF~70910-900, Brazil}
\email{pavel@unb.br}


\begin{abstract}
We say that an element $g$ of a group $G$ is almost right Engel if there is a finite set ${\mathscr R}(g)$ such that for every $x\in G$ all sufficiently long commutators $[...[[g,x],x],\dots ,x]$ belong to ${\mathscr R}(g)$, that is, for every $x\in G$ there is a positive integer $n(x,g)$ such that $[...[[g,x],x],\dots ,x]\in {\mathscr R}(g)$ if $x$ is repeated at least $n(x,g)$ times. Thus, $g$ is a right Engel element precisely when we can choose ${\mathscr R}(g)=\{ 1\}$.

We prove that if all elements of a compact (Hausdorff) group $G$ are almost right Engel, then
$G$ has a finite normal subgroup $N$ such that $G/N$ is locally nilpotent. If in addition there is a uniform bound $|{\mathscr R}(g)|\leq m$ for the orders of the corresponding sets, then the subgroup $N$ can be chosen of order bounded in terms of $m$. The proofs use the Wilson--Zelmanov theorem saying that Engel profinite groups are locally nilpotent and previous results of the authors about compact groups all elements of which are almost left Engel.
\end{abstract}
\maketitle

\section{Introduction}

A group $G$ is called an Engel group if for every $x,g\in G$ the equation $[x, {}_ng]=1$ holds for some $n=n(x,g)$. (Throughout the paper, we use the
abbreviated notation for the Engel commutator
$[a,{}_nb]=[a,b,b,\dots ,b]$, where $b$ is repeated $n$ times.)
Recall that a group is said to be locally nilpotent if every finite subset generates a nilpotent subgroup. Clearly, any locally nilpotent group is an Engel group. Wilson and Zelmanov \cite{wi-ze} proved the converse for profinite groups: any Engel profinite group is locally nilpotent. Later Medvedev \cite{med} extended this result to Engel compact (Hausdorff) groups.

In \cite{khu-shu162} we considered compact groups all of whose elements are almost left Engel in the following precise sense.
 An element $g$ of a group $G$ is \emph{almost left Engel} if there is a \emph{finite} set ${\mathscr E}(g)$ such that
 for every $x\in G$ there is a positive integer $l(x,g)$ such that
 $$
 [x,{}_ng]\in {\mathscr E}(g)\qquad \text{for all }n\geq l(x,g).
 $$
Thus, $g$ is a left Engel element precisely when we can choose ${\mathscr E}(g)=\{ 1\}$, and if we can choose ${\mathscr E}(g)=\{ 1\}$ for all $g\in G$, then $G$ is an Engel group.

We proved in \cite{khu-shu162} that compact groups all elements of which are almost left Engel are finite-by-(locally nilpotent). By a compact group we mean a compact Hausdorff topological group. In the case where there is a uniform bound $m$ for the cardinalities of all the sets ${\mathscr E}(g)$, the finite subgroup in question can be chosen of order bounded in terms of $m$ (and this quantitative result makes sense also for finite groups.)

In this paper we prove similar results for compact groups all elements of which are almost right Engel in the following sense.
An element $g$ of a group $G$ is \emph{almost right Engel} if there is a \emph{finite} set ${\mathscr R}(g)$ such that for every $x\in G$
there is a positive integer $r(x,g)$ such that
 $$[g,{}_nx]\in {\mathscr R}(g)\qquad \text{for all }n\geq r(x,g).
 $$
Thus, $g$ is a right Engel element precisely when we can choose ${\mathscr R}(g)=\{ 1\}$, and if we can choose ${\mathscr R}(g)=\{ 1\}$ for all $g\in G$, then $G$ is an Engel group.

\begin{theorem}\label{t-e}
Suppose that $G$ is a compact group all elements of which are almost right Engel.
Then $G$ has a finite normal subgroup $N$ such that $G/N$ is locally nilpotent.
\end{theorem}

In Theorem~\ref{t-e} it also follows that there is a locally nilpotent subgroup of finite index -- just consider $C_G( N)$. If there is a uniform bound $m$ for the cardinalities of the sets ${\mathscr R}(g)$ in Theorem~\ref{t-e}, then the subgroup $N$ in the conclusion can be chosen to be of order bounded in terms of $m$ (Corollary~\ref{c-em}).

It is well-known that the inverse of a right Engel element is left Engel. But
it is unclear if an almost  right Engel element must be almost left Engel. It is only by virtue of our Theorem~\ref{t-e} that if all elements of a compact group are almost right Engel, then all elements are also almost left Engel (the converse also holds by the main result of \cite{khu-shu162}).

 The proof uses the Wilson--Zelmanov theorem on profinite Engel groups and the aforementioned results of the authors about compact groups all elements of which are almost left Engel. First the case of a finite group $G$ is considered, where the result must be quantitative: if there is a uniform bound $ |{\mathscr R}(g)|\leq m$ for all $g\in G$, then the order of the nilpotent residual $\gamma _{\infty}(G)=\bigcap _i\gamma _i(G)$ is bounded in terms of $m$ only (Theorem~\ref{t-finite}). Then Theorem~\ref{t-e} is proved for profinite groups. Finally, the result for compact groups is derived with the use of the structure theorems for compact groups.

First in \S\,\ref{s-1} we discuss some elementary properties of minimal subsets ${\mathscr E}(g)$ and ${\mathscr R}(g)$ in the definitions of almost left Engel and almost right Engel elements. We deal with finite groups in \S\,\ref{s-f}, with profinite groups in \S\,\ref{s-pf}, and consider the general case of compact groups in \S\,\ref{s-comp}.

Our notation and terminology is standard; for profinite groups, see, for example, \cite{wilson}.
A~subgroup (topologically) generated by a subset $S$ is denoted by $\langle S\rangle$. For a group $A$ acting by automorphisms on a group $B$ we use the usual notation for commutators $[b,a]=b^{-1}b^a$ and $[B,A]=\langle [b,a]\mid b\in B,\;a\in A\rangle$, and for centralizers $C_B(A)=\{b\in B\mid b^a=b \text{ for all }a\in A\}$ and $C_A(B)=\{a\in A\mid b^a=b\text{ for all }b\in B\}$.

Throughout the paper we shall write, say, ``$(a,b,\dots )$-bounded'' to abbreviate
``bounded above in terms of $a, b,\dots $ only''.

\section{Properties of Engel sinks}
\label{s-1}

If $g$ is an almost left Engel element of a group $G$, then there is a finite
set ${\mathscr E}(g)$ such that for every $x\in G$
there is a positive integer $l(x,g)$ such that
\begin{equation} \label{e-def}
[x,{}_ng]\in {\mathscr E}(g)\qquad \text{for any }n\geq l(x,g).
\end{equation}
If ${\mathscr E}'(g)$ is another set with the same property for possibly different numbers $l'(x,g)$, then ${\mathscr E}(g)\cap {\mathscr E}'(g)$ also satisfies the same condition with the numbers $l''(x,g)=\max\{l(x,g),l'(x,g)\}$. Hence there is a \emph{minimal} set satisfying the above condition, which we again denote by ${\mathscr E}(g)$ and call the \emph{left Engel sink of $g$}. Henceforth we shall always use the notation ${\mathscr E}(g)$ to denote the (minimal) finite left Engel sink of $g$ when it exists, and $l(x,g)$ the corresponding minimum numbers satisfying \eqref{e-def}. In cases where we need to consider the left Engel sink constructed with respect to a subgroup $H$ containing $g$, we write ${\mathscr E}_H(g)$.

Recall the characterization of ${\mathscr E}(g)$ obtained by considering the mapping $z\mapsto [z,g]$ for $z\in {\mathscr E}(g)$.

\begin{lemma}[{\cite[Lemma~2.1]{khu-shu162}}]\label{l-sink}
If an element $g$ of a group $G$ has a finite left Engel sink, then
$${\mathscr E}(g)=\{z\mid z=[z,{}_ng]\quad\text{for some }\; n\geq 1\}.$$
\end{lemma}

If an element $g$ of a group $G$ has a finite left Engel sink, then, by Lemma~\ref{l-sink}, for a subgroup $H$ containing $g$ we have ${\mathscr E}_H(g)= {\mathscr E}(g)\cap H$, and if $N$ is a normal subgroup of $G$, then for $\bar g=gN$ the left Engel sink ${\mathscr E}_{G/N}(\bar g)$ is the image of ${\mathscr E}(g)$ in $G/N$. These properties will be used throughout the paper without special references.

Similar observations can be made about right Engel sinks. Suppose that for an element $g$ of a group $G$ there is a finite
set ${\mathscr R}(g)$ such that for every $x\in G$
there is a positive integer $r(x,g)$ such that
\begin{equation} \label{r-def}
[g,{}_n x]\in {\mathscr R}(g)\qquad \text{for any }n\geq r(x,g).
\end{equation}
If ${\mathscr R}'(g)$ is another set with the same property for possibly different numbers $r'(x,g)$, then ${\mathscr R}(g)\cap {\mathscr R}'(g)$ also satisfies the same condition with the number $r''(x,g)=\max\{r(x,g),r'(x,g)\}$. Hence there is a \emph{minimal} set satisfying the above condition, which we again denote by ${\mathscr R}(g)$ and call the \emph{right Engel sink of $g$}. Henceforth we shall always use the notation ${\mathscr R}(g)$ to denote the (minimal) finite right Engel sink of $g$ when it exists, and $r(x,g)$ the corresponding minimum numbers satisfying \eqref{e-def}. In cases where we need to consider the right Engel sink constructed with respect to a subgroup $H$ containing $g$, we write ${\mathscr R}_H(g)$

\begin{lemma}\label{l-r-sink}
If an element $g$ of a group $G$ has a finite right Engel sink, then
\begin{equation}\label{r-sink}{\mathscr R}(g)=\{z\mid z=[g,{}_nx]=[g,{}_{n+m}x],\quad \; x\in G,\;\; n\geq 1,\;\;m\geq 1\}.
\end{equation}
\end{lemma}
Of course, elements $x$ and numbers $m,n$ in \eqref{r-sink} vary for different $z$ and are not unique.

\bp
Clearly, elements \eqref{r-sink} belong to ${\mathscr R}(g)$. Furthermore, the set of all elements \eqref{r-sink} forms (some) right Engel sink of $g$.
Indeed, for any $x\in G$, we have $[g,{}_sx]\in {\mathscr R}(g)$ for some $s\geq 1$. As $s$ increases, the corresponding commutators remain in ${\mathscr R}(g)$, and since this set is finite, there must be repeats like $[g,{}_nx]=[g,{}_{n+m}x]$ for some $n\geq s$ and $m\geq 1$. Therefore the set of elements \eqref{r-sink} is equal to ${\mathscr R}(g)$ due to the minimality.
\ep

If an element $g$ of a group $G$ has a finite right Engel sink, then, by Lemma~\ref{l-r-sink}, for a subgroup $H$ containing $g$ we have ${\mathscr R}_H(g)\subseteq {\mathscr R}(g)\cap H$, and if $N$ is a normal subgroup of $G$, then for $\bar g=gN$ the right Engel sink ${\mathscr R}_{G/N}(\bar g)$ is equal to the image of ${\mathscr R}(g)$ in $G/N$. These properties will be used throughout the paper without special references.

\begin{lemma}\label{l-c-s}
Suppose that an element $g$ of a group $G$ has a finite right Engel sink of cardinality $|{\mathscr R}(g)|=m$. If $h\in C_G(g)$, then $h^{m!}$ centralizes
${\mathscr R}(g)$.
\end{lemma}

\bp
For an element $h$ of the centralizer $C_G(g)$,
the equation $z=[g,{}_nx]=[g,{}_{n+m}x]$ implies $z^h=[g^h,{}_nx^h]=[g^h,{}_{n+m}x^h]=[g,{}_nx^h]=[g,{}_{n+m}x^h]$. By Lemma~\ref{l-r-sink} it follows that
${\mathscr R}(g)$
is invariant under conjugation by $h$. Hence the result.
\ep

The following lemma was proved by Heineken \cite{hei}.

\begin{lemma}[{\cite[12.3.1]{rob}}]\label{l-hei}
 If $g$ is a right Engel element of a group $G$, then $g^{-1}$ is a left Engel element.
\end{lemma}

 We shall also need the following formula from the proof of Lemma~\ref{l-hei}:
 \begin{equation}\label{r-l}
[x,{}_{n+1}g]=[g^{-1},{}_ng^{x^{-1}}]^{xg}
\end{equation}
 for any elements $x,g$ of any group $G$ and any positive integer $n$.

We now use this formula for a metabelian group in the following lemma, in which,  for greater generality, we consider not necessarily minimal  left and right Engel sinks.

\begin{lemma}\label{l-metab}
 If $G$ is a metabelian group, then a right Engel sink of the inverse $g^{-1}$ of an element $g\in G$ is a left Engel sink of $g$. In particular, if $g^{-1}$ has a finite right Engel sink ${\mathscr R}(g^{-1})$, then $g$ has a finite left Engel sink and ${\mathscr E}(g)\subseteq {\mathscr R}(g^{-1})$.
\end{lemma}

\bp
For any $x\in G$ and $n\geq 1$, we transform formula \eqref{r-l} using the fact that $[x,g]$ commutes with any commutator:
\begin{align*}
[x,{}_{n+2}g]=[[x,g],{}_{n+1}g]&
=[g^{-1},{}_ng^{[x,g]^{-1}}]^{[x,g]g}\\
&= [g^{-1},{}_ng^{[x,g]^{-1}}]^{g}\\
&= [g^{-1},{}_ng^{[x,g]^{-1}g}].
\end{align*}
For all $n$ larger that some number depending on $x$ and $g$, the right-hand side belongs to a given right Engel sink of $g^{-1}$, which is therefore also a left Engel sink of $g$ in view of the left-hand side. The second statement obviously follows.
\ep

Another lemma for metabelian groups was proved before.

\begin{lemma}[{\cite[Lemma~2.3]{khu-shu162}}]\label{l-metab2} Let $G$ be a metabelian group and suppose that an element $g\in G$ has a finite left Engel sink. Then

{\rm (a)} ${\mathscr E}(g)$ is a normal subgroup contained in $G'$;

{\rm (b)} elements of ${\mathscr E}(g)$ in the same orbit under the map $z\to [z,g]$ have the same order.
\end{lemma}

\section{Finite groups}
\label{s-f}

Of course, in any finite group $G$ every element $g\in G$ has finite (minimal) right Engel sink ${\mathscr R}(g)$. A meaningful result must be of quantitative nature, and this is what we prove in this section. The following theorem will also be used in the proof of the results on profinite and compact groups.

\begin{theorem}\label{t-finite}
Let $G$ be a finite group, and $m$ a positive integer. Suppose that for every $g\in G$ the cardinality of the right Engel sink ${\mathscr R}(g)$ is at most $m$.
 Then $G$ has a normal subgroup $N$ of order bounded in terms of $m$ such that $G/N$ is nilpotent.
\end{theorem}

The conclusion of the theorem can also be stated as a bound in terms of $m$ for the order of the nilpotent residual subgroup $\gamma _{\infty}(G)$, the intersection of all terms of the lower central series (which for a finite group is of course also equal to some subgroup $\gamma _n(G)$).

First we recall or prove a few preliminary results. We shall use the following well-known properties of coprime actions: if $\alpha $ is an automorphism of a finite
group $G$ of coprime order, $(|\alpha |,|G|)=1$, then $ C_{G/N}(\alpha )=C_G(\alpha )N/N$ for any $\alpha $-invariant normal subgroup $N$, the equality $[G,\alpha ]=[[G,\alpha ],\alpha ]$ holds, and if $G$ is in addition abelian, then $G=[G,\alpha ]\times C_G(\alpha )$.

\begin{lemma}\label{l-Vg}
 If $V$ is an abelian subgroup of a finite group $G$, and $g\in G$ an element of coprime order normalizing $V$, then $$[V,g]={\mathscr E}_{V\langle g\rangle}(g)={\mathscr R}_{V\langle g\rangle}(g).$$
 \end{lemma}

\bp
We have $C_{[V,g]}(g)=1$ because the action of $g$ on $V$ is coprime. Then $[V,g]=\{[v,g]\mid v\in [V,g]\}$ and therefore also $ [V,g]=\{[v,{}_ng]\mid v\in [V,g]\}$ for any $n$. Hence, $[V,g]$ is contained in ${\mathscr E}_{V\langle g\rangle}(g)$, which is contained in ${\mathscr R}_{V\langle g\rangle}(g^{-1})$ by Lemma~\ref{l-metab}. By Lemma~\ref{l-r-sink} we also clearly have ${\mathscr R}_{V\langle g\rangle}(g^{-1})\subseteq [V,g^{-1}]=[V,g]$. As a result,
 $[V,g]={\mathscr E}_{V\langle g\rangle}(g)={\mathscr R}_{V\langle g\rangle}(g^{-1})$.
 Since $[V,g^{-1}]=[V,g]$, we finally get $[V,g]={\mathscr E}_{V\langle g\rangle}(g)={\mathscr R}_{V\langle g\rangle}(g)$.
\ep

\begin{lemma}\label{l0}
Let $P$ be a $p$-subgroup of a finite group $G$, and $g\in G$ a $p'$-element normalizing $P$. Then the order of $[P,g]$ is bounded in terms of the cardinality of the right Engel sink ${\mathscr R}(g)$.
\end{lemma}

\begin{proof}
For the abelian $p$-group $V=[P,g]/[P,g]'$ we have $V=[V,g]=
{\mathscr R}_{V\langle g\rangle}(g)$ by Lemma~\ref{l-Vg}. Thus, $|V|\leq |{\mathscr R}(g)|$.

Since $[P,g]$ is a nilpotent group, its order is bounded in terms of $|V|=|[P,g]/[P,g]'|$ and the nilpotency class of $[P,g]$. The proof will be complete if we show that, as a crude bound, the nilpotency class of $[P,g]$ is at most $2|{\mathscr R}(g)|+1$.

Let $\gamma_i$ denote the terms of the lower central series of $[P,g]$. We claim that the number of factors of the lower central series of $[P,g]$ on which $g$ acts nontrivially is at most $|{\mathscr R}(g)|$. Indeed,
for any such a factor $U=\gamma _i/\gamma _{i+1}$ we have $
1\ne [U,g]= {\mathscr R}_{U\langle g\rangle}(g)$. As a result, $\gamma_i\setminus \gamma _{i+1}$ contains an element of ${\mathscr R}(g)$, which proves the claim.

It remains to observe that $g$ cannot act trivially on two consecutive nontrivial factors of the lower central series of $[P,g]$. Namely, if
$[\gamma_i, g]\leq \gamma _{i+1}$ and $[\gamma_{i+1}, g]\leq \gamma _{i+2}$, then by the Three Subgroup Lemma the inclusions $[\gamma_i, g, [P,g]]\leq [\gamma _{i+1}, [P,g]]=\gamma _{i+2}$ and $[[P,g],\gamma_{i}, g]=[ \gamma _{i+1},g]\leq \gamma _{i+2}$ imply the inclusion $[g, [P,g],\gamma _i]=[[P,g],\gamma _i]=\gamma _{i+1}\leq \gamma _{i+2}$. The last inclusion implies that $\gamma_{i+1}=1$, since the group is nilpotent.
\end{proof}

We reproduce a lemma on coprime actions from \cite{khu-shu162}.

\begin{lemma}[{\cite[Lemma~3.3]{khu-shu162}}]\label{l2}
Let $V$ be an elementary abelian $p$-group, and $U$ a $p'$-group of
automorphisms of $V$. If $|[V,u]|\leq m$ for every $u\in U$, then $|[V,U]|$ is $m$-bounded, and therefore $|U|$ is also $m$-bounded.
\end{lemma}

 Recall that the Fitting series starts with the Fitting subgroup $F_1(G)=F(G)$, and by induction, $F_{k+1}(G)$ is the inverse image of $F(G/F_k(G))$. If $G$ is a soluble group, then the least number $h$ such that $F_h(G)=G$ is the \textit{Fitting height} of $G$. The following lemma is well known and is easy to prove (see, for example, \cite[Lemma~10]{khu-maz}).

 \begin{lemma}\label{l-metan}
 If $G$ is a finite group of Fitting height 2, then $\gamma _{\infty}(G)=\prod _q [F_q,G_{q'}]$, where $F_q$ is a Sylow $q$-subgroup of $F(G)$, and $G_{q'}$ is a Hall ${q'}$-subgroup of $G$.
 \end{lemma}

We now approach the proof of Theorem~\ref{t-finite} with the following two lemmas, dealing separately with soluble and simple groups.

\begin{lemma}\label{l3}
If $G$ is a finite soluble group such that $|{\mathscr R}(g)|\leq m$ for all $g\in G$, then the exponent of $G/F(G)$ divides $m!$.
\end{lemma}

\begin{proof}
Since
$$F(G)=\bigcap_p O_{p',p}(G),$$
it is sufficient to obtain a bound for the exponent of $G/O_{p',p}(G)$ for every prime $p$ dividing $|G|$. For a fixed such prime $p$, it is sufficient to obtain a bound in terms of $m$ for the orders of both $p'$- and $p$-elements of $G/O_{p',p}(G)$.

If $g$ is a $p'$-element of $G/O_{p',p}(G)$, then $g$ acts faithfully by conjugation on the Frattini quotient $V$ of $O_{p',p}/O_{p'}(G)$. By Lemma~\ref{l-Vg} we have $[V,g]={\mathscr R}_{V\langle g\rangle}(g)$ and therefore, $|[V,g]|\leq |{\mathscr R}(g)|\leq m$. Since $g$ acts faithfully on $[V,g]$, the order $|g|$ must divide $m!$.

Now let $g$ be a $p$-element of $G/O_{p',p}(G)$. Since $O_p(G/O_{p',p}(G))=1$, the group $\langle g\rangle $ acts faithfully on a Sylow $q$-subgroup $Q$ of $F(G/O_{p',p}(G))$ for at least one prime $q\ne p$. Then $\langle g\rangle $ also acts faithfully on the Frattini quotient $W=Q/\Phi (Q)$. By the same argument as above, $|[W,g]|\leq m$, and since $g$ acts faithfully on $[W,g]$, the order $|g|$ must divide $m!$.
\ep

\begin{lemma}\label{l-simple}
If $G$ is a finite non-abelian simple group such that $|{\mathscr R}(g)|\leq m$ for all $g\in G$, then the order of $G$ is bounded above in terms of $m$.
\end{lemma}

\bp
By Lemma~\ref{l0}, for any $p$-subgroup $P$ and a $p'$-element $g\in N_G(P)$, the order of $[P,g]$ is $m$-bounded. It follows from the classification and the structure of finite simple groups that this condition implies that the order of $G$ is $m$-bounded.

Indeed, we can assume that $G$ is either an alternating group or a group of Lie type. For an alternating group $G=A_n$ it is easy to see that $n$ is $m$-bounded. For example, in the semidirect product $A\langle b\rangle$ of an elementary abelian group $A$ of order $3^r$ and its group of automorphisms generated by an element $ b$ of order $2$ inverting all elements of $A$, we have $|{\mathscr R}(b)|=|A|=3^r$ by Lemma~\ref{l-Vg}, and $A\langle b\rangle$ embeds in $S_{3^r}$ with point stabilizer $\langle b\rangle$, and therefore in $A_{2\cdot 3^r}$. If $|{\mathscr R}(g)|\leq m$ for all $g\in A_n$, let $r$ be the least integer such that $m< 3^r$. Then we must have $n< 2\cdot 3^r$, in particular, $n<2\cdot 3^{[\log_3m]+1}=6m$.

Now let $G$ be a finite simple group of Lie type $G=L_n(\mathbb{F} _{q})$ of degree $n$ over a field of order $q$. Clearly, it is sufficient to show that both $n$ and $q$ are $m$-bounded. To obtain a bound for $n$, it suffices to consider the Weyl group, which, for large $n$, contains a subgroup isomorphic to a symmetric group of large degree, which in turn contains subgroups of type $[P,g]$ of large order as shown in the preceding paragraph.

To obtain a bound for $q=p^k$ in terms of $m$, it is sufficient to show that $G$ has a $p$-subgroup $P$ and a $p'$-element $g\in N_G(P)$ such that $[P,g]$ contains a subgroup isomorphic to the additive group of the field $\mathbb{F} _{q}$. This follows from the well-known facts about simple groups of Lie type. For example, we can use the fact that $G=L_n(\mathbb{F} _{q})$ either contains a subgroup isomorphic to $SL_2(q)$ or $PSL_2(q)$, or is a Suzuki group over $\mathbb{F} _{q}$. (Even stronger statements are proved in \cite{Liebeck--Nikolov--Shalev}). In $SL_2(p^k)$, put
$$
g=\begin{pmatrix} \zeta ^{-1}&0\\0&\zeta \end{pmatrix},
$$
where $\zeta $ is a nontrivial $p'$-element of the multiplicative group of the field $\mathbb{F} _{q}$ such that $\zeta ^2\ne 1$. (The latter condition can always be satisfied for $k>1$, or for $p>3$ if $k=1$.) This element normalizes and acts fixed-point-freely on the abelian $p$-subgroup of upper-triangular matrices
$$
T=\left\{\left.\begin{pmatrix} 1&a\\0&1\end{pmatrix}\right| a\in \mathbb{F} _{q}\right\},
$$
which is isomorphic to the additive group of $\mathbb{F} _{q}$. Since $T=[T,g]$, it follows that $q$ is $m$-bounded. In the quotient $PSL_2(p^k)$ of $SL_2(p^k)$ by the centre, the image of $T$ is isomorphic to $T$. Finally, the case of $G$ being a Suzuki group is dealt with in similar fashion, by considering the action of a diagonal $2'$-element on a Sylow $2$-subgroup. Thus, $q$ is $r$-bounded.
\ep

We are now ready to prove Theorem~\ref{t-finite}.

\begin{proof}[Proof of Theorem~\ref{t-finite}] Recall that $G$ is a finite group such that
$|{\mathscr R}(g)|\leq m$ for every $g\in G$. We need to show that $|\gamma _{\infty }(G)|$ is $m$-bounded.

First suppose that $G$ is soluble. Since $G/F(G)$ has $m$-bounded exponent by Lemma~\ref{l3}, the Fitting height of $G$ is $m$-bounded by the Hall--Higman theorems \cite{ha-hi}. Hence we can use induction on the Fitting height, with trivial base when the group is nilpotent and $\gamma _{\infty }(G)=1$. When the Fitting height is at least 2, consider the second Fitting subgroup $F_2(G)$. By Lemma \ref{l-metan} we have $\gamma _{\infty }(F_2(G))=\prod _q [F_q,H_{q'}]$, where $F_q$ is a Sylow $q$-subgroup of $F(G)$, and $H_{q'}$ is a Hall ${q'}$-subgroup of $F_2(G)$, the product taken over prime divisors of $|F(G)|$. For a given $q$, let $\bar H_{q'}=H_{q'}/C_{H_{q'}}(F_q)$, and let $V$ be the Frattini quotient $F_q/\Phi (F_q)$. Note that $\bar H_{q'}$ acts faithfully on $V$, since the action is coprime \cite[Satz~III.3.18]{hup}.

For every $x\in \bar H_{q'}$ the order $|[V,x]|$ is $m$-bounded by Lemma~\ref{l0}. Then $|\bar H_{q'}|$ is $m$-bounded by Lemma~\ref{l2}. As a result, $|[F_q , H_{q'}]|= |[F_q ,\bar H_{q'}]|$ is $m$-bounded, since $[F_q ,\bar H_{q'}]$ is the product of $m$-boundedly many subgroups $[F_q ,\bar h]$ for $h\in H_{q'}$, each of which has $m$-bounded order by Lemma~\ref{l0}.

For the same reasons, there are only $m$-boundedly many primes $q$ for which $[F_q , H_{q'}]\ne 1$. As a result, $|\gamma _{\infty }(F_2(G))|$ is $m$-bounded. Induction on the Fitting height applied to $G/\gamma _{\infty }(F_2(G))$ completes the proof in the case of soluble $G$.

Now consider the general case. First we show that the quotient $G/R(G)$ by the soluble radical is of $m$-bounded order.
Let $E$ be the socle
of $G/R(G)$. It is known that $E$ contains its centralizer in $G/R(G)$, so it suffices to show that $E$ has $m$-bounded order. In the quotient by the soluble radical, $E=S_1\times\dots\times S_k$ is a direct product of non-abelian finite simple groups $S_i$. By Lemma~\ref{l-simple}, every $S_i$ has $m$-bounded order, and it remains to show that the number of factors is also $m$-bounded. By Schmidt's theorem \cite[Satz~III.5.1]{hup}, every $S_i$ has a non-nilpotent soluble subgroup $R_i$, for which $\gamma _{\infty} (R_i)\ne 1$. Applying the already proved theorem for soluble groups to $T=R_1\times \dots \times R_k$ we obtain that $|\gamma _{\infty} (T)|$ is $m$-bounded, whence the number of factors is $m$-bounded.

Thus, $|G/R(G)|$ is $m$-bounded. Let $g\in G$ be an arbitrary element. The subgroup $R(G)\langle g\rangle$ is soluble, and therefore $|\gamma _{\infty }(R(G)\langle g\rangle)|$ is $m$-bounded by the above. Since $\gamma _{\infty }(R(G)\langle g\rangle)$ is normal in $R(G)$, its normal closure $\langle \gamma _{\infty }(R(G)\langle g\rangle) ^G\rangle$ is a product of at most $|G/R(G)|$ conjugates, each normal in $R(G)$, and therefore has $m$-bounded order. Choose a transversal $\{t_1,\dots, t_k\}$ of $G$ modulo $R(G)$ and set
$$
K=\prod _i\langle \gamma _{\infty }(R(G)\langle t_i\rangle) ^G\rangle,
$$
which is a normal subgroup of $G$ of $m$-bounded order. It is sufficient to obtain an $m$-bounded estimate for $|\gamma _{\infty }(G/K)|$. Hence we can assume that $K=1$. Then
\begin{equation*}
[R(G), g, \dots , g]=1\qquad \text{for any } g\in G,
 \end{equation*}
 when $g$ is repeated sufficiently many times. Indeed, $g\in R(G)t_i$ for some $t_i$, and the subgroup $R(G)\langle t_i\rangle$ is nilpotent due to our assumption that $K=1$.

This means that $R(G)$ consists of right Engel elements. By Baer's theorem \cite[12.3.7]{rob}, therefore $R(G)$ is contained in the hypercentre of $G$ equal to some term of the upper central series $\zeta _ i(G)$. Hence the index $|G:\zeta _ i(G)|$ is $m$-bounded. By a quantitative version of Baer's theorem \cite[14.5.1]{rob}, the order $|\gamma _{\infty}(G)|$ is also $m$-bounded. (A~quantitative version of this theorem of Baer can be extracted from the original proof in \cite{baer} or \cite[14.5.1]{rob}, see also a remark after Theorem~2.2 in \cite{fer-mor}.)
\end{proof}

\section{Profinite groups}\label{s-pf}

In this and the next sections, unless stated otherwise, a subgroup of a topological group will always mean a closed subgroup, all homomorphisms will be continuous, and quotients will be by closed normal subgroups. This also applies to taking commutator subgroups, normal closures, subgroups generated by subsets, etc. Of course, any finite subgroup is automatically closed. We also say that a subgroup is generated by a subset $X$ if it is generated by $X$ as a topological group.

In this section we prove Theorem~\ref{t-e} for profinite groups, while Corollary~\ref{c-em} for profinite groups is an immediate corollary of Theorem~\ref{t-finite}.

\begin{theorem}\label{t-eprof}
Suppose that $G$ is a profinite group all elements of which are almost right Engel.
Then $G$ has a finite normal subgroup $N$ such that $G/N$ is locally nilpotent.
\end{theorem}

Recall that pronilpotent groups are defined to be inverse limits of finite nilpotent groups.

\begin{lemma}\label{l-p-n}
If all elements of a pro\-nil\-po\-tent group $G$ are almost right Engel, then the group $G$ is locally nilpotent.
\end{lemma}

\begin{proof}
For any $g\in G$, since ${\mathscr R}(g)$ is finite, there is
an open normal subgroup $N$ with nilpotent quotient $G/N$ such that ${\mathscr R}(g)\cap N=\{1\}$. On the other hand, ${\mathscr R}(g)\subset N$, since $G/N$ is nilpotent. Thus, ${\mathscr R}(g)=\{1\}$ for every $g\in G$, which means that $G$ is
an Engel profinite group. Then $G$ is locally nilpotent by the Wilson--Zelmanov theorem \cite[Theorem~5]{wi-ze}.
\end{proof}

 Recall that the pro\-nil\-po\-tent residual of a profinite group $G$ is $\gamma _{\infty}(G)=\bigcap _i\gamma _i(G)$, where $\gamma _i(G)$ are the terms of the lower central series; this is the smallest normal subgroup with pro\-nil\-po\-tent quotient. The following lemma is well known and is easy to prove. Here, element orders are understood as Steinitz numbers. The same results also hold in the special case of finite groups.

 \begin{lemma}[{\cite[Lemma~4.3]{khu-shu162}}]\label{l-res}
 {\rm (a)} The pro\-nil\-po\-tent residual $\gamma _{\infty}(G)$ of a profinite group $G$ is equal to the subgroup generated by all commutators $[x,y]$, where $x,y$ are elements of coprime orders.

 {\rm (b)} For any normal subgroup $N$ of a profinite group $G$ we have $\gamma _{\infty}(G/N)=
 \gamma _{\infty}(G)N/N$.
 \end{lemma}

The following generalization of Hall's criterion for nilpotency \cite{hall58}, which will be used later, already appeared in \cite{khu-shu162}.
The derived subgroup of a group $B$ is denote by $B'$.

\begin{lemma}[{\cite[Proposition~4.4(b)]{khu-shu162}}]\label{p-hall}
 Suppose that $B$ is a normal subgroup of a profinite group $A$ such that $B$ is pro\-nil\-po\-tent and $\gamma _{\infty}(A/B')$ is finite.
Then the subgroup $D=C_A(\gamma _{\infty }(A/B'))=\{a\in A\mid [\gamma _{\infty }(A), a]\leq B'\}$ is open and pro\-nil\-po\-tent.
\end{lemma}

Recall that in Theorem~\ref{t-eprof} we need to show that there is a finite normal subgroup with locally nilpotent quotient. The first step is to prove the existence of an open locally nilpotent subgroup.

\begin{proposition}\label{p-pf1}
If $G$ is a profinite group all elements of which are almost right Engel, then $G$ has an open
normal pro\-nil\-po\-tent subgroup.
\end{proposition}

Of course, the subgroup in question will also be locally nilpotent by Lemma~\ref{l-p-n}; the result can also be stated as the openness of the largest normal pro\-nil\-po\-tent subgroup.

\begin{proof}
For every $g\in G$ we choose an open normal subgroup $N_g$ such that ${\mathscr R}(g^{-1})\cap N_g=\{1\}$. Then $g^{-1}$ is a right Engel element in $N_g\langle g\rangle$, and therefore, $g$ is a left Engel element in $N_g\langle g\rangle$ by Lemma~\ref{l-hei}. By Baer's theorem \cite[Satz~III.6.15]{hup}, in every finite quotient of $N_g\langle g\rangle$ the image of $g$ belongs to the Fitting subgroup. As a result, the subgroup $[N_g, g]$ is pro\-nil\-po\-tent.

Let $\tilde N_g$ be the normal closure of $[N_g, g]$ in $G$. Since $[N_g, g]$ is normal in the subgroup $N_g$ of finite index, $[N_g, g]$ has only finitely many conjugates, so $\tilde N_g$ is a product of finitely many normal subgroups of $N_g$, each of which is pro\-nil\-po\-tent.
Hence, so is $\tilde N_g$. Therefore all the subgroups $\tilde N_g$ are contained in the largest normal pro\-nil\-po\-tent subgroup $K$.

The quotient $G/K$ is an $FC$-group (that is, every conjugacy class is finite), since every element $\bar g\in G/K$ is centralized by the image of $N_g$, which has finite index in $G$. A~profinite $FC$-group has finite derived subgroup \cite[Lemma~2.6]{sha}. Hence we can choose an open subgroup of $G/K$ that has trivial intersection with the finite derived subgroup of $G/K$ and therefore is abelian; let $H$ be its full inverse image in~$G$. Thus, $H$ is an open subgroup such that the derived subgroup $H'$ is contained in~$K$.

In the metabelian quotient $M=H/K'$, all elements have finite right Engel sinks, as this property is inherited from the group $G$. By Lemma~\ref{l-metab} it follows that all elements of $M$ also have
finite left Engel sinks. Therefore we can apply Theorem~4.1 of \cite{khu-shu162}, by which $M$ is finite-by-(locally nilpotent), that is, $\gamma _{\infty}(M)$ is finite.

Let $W$ be the full inverse image of $M$, which is an open subgroup of $G$ containing $K$, and let $ \Gamma$ be the full inverse image of $\gamma _{\infty}( M)$.
Now let $F=C_W(\gamma _{\infty}( M))=\{w\in W\mid [\Gamma ,w]\leq K'\}$. By Lemma~\ref{p-hall}(b), this is an open normal pro\-nil\-po\-tent subgroup, which completes the proof of Proposition~\ref{p-pf1}.
\end{proof}

\begin{proof}[Proof of Theorem~\ref{t-eprof}] Recall that $G$ is a profinite group all elements of which are almost right Engel, and we need to show that $\gamma _{\infty}(G)$ is finite. Henceforth we denote by $F(L)$ the largest normal pro\-nil\-po\-tent subgroup of a profinite group $L$. By Proposition~\ref{p-pf1} we already know that $G$ has an open normal pro\-nil\-po\-tent
 subgroup, so that $F(G)$ is also open.

Since $G/F(G)$ is finite, we can use induction on $|G/F(G)|$.
 The basis of this induction includes the trivial case $G/F(G)=1$ when $\gamma _{\infty}(G)=1$. But the bulk of the proof deals with the case where $G/F(G)$ is a finite simple group.

Thus, we assume that $G/F(G)$ is a finite simple group (abelian or non-abelian).
Let $p$ be a prime divisor of $|G/F(G)|$, and $g\in G\setminus F(G)$ an element of order $p^n$, where $n$ is either a positive integer or $\infty$ (so $p^n$ is a Steinitz number). For any prime $q\ne p$, the element $g$ acts by conjugation on the Sylow $q$-subgroup $Q$ of $F(G)$ as an automorphism of order dividing $p^n$. The subgroup $[Q, g]$ is a normal subgroup of $Q$ and therefore also a normal subgroup of $F(G)$. The image of $[Q, g]$ in any finite quotient has order bounded in terms of $|{\mathscr R}(g)|$ by Lemma~\ref{l0}. It follows that $[Q, g]$ is finite of order bounded in terms of $|{\mathscr R}(g)|$.

Since $[Q, g]$ is normal in $F(G)$, its normal closure $\langle [Q, g]^G\rangle $ in $G$ is a product of finitely many conjugates and is therefore also finite. Let $R$ be the product of these closures $\langle [Q, g]^G\rangle $ over all Sylow $q$-subgroups $Q$ of $F(G)$ for $q\ne p$. Since $[Q, g]$ is finite of order bounded in terms of $|{\mathscr R}(g)|$ as shown above, there are only finitely many primes $q$ such that $[Q,g]\ne 1$ for the Sylow $q$-subgroup $Q$ of $F(G)$. Therefore $R$ is finite, and it is sufficient to prove that $\gamma _{\infty }(G/R)$ is finite. Thus, we can assume that $R=1$. Note that then $[Q, g^a]=1$ for any conjugate $g^a$ of $g$ and any Sylow $q$-subgroup $Q$ of $F(G)$ for $q\ne p$.

Choose a transversal $\{t_1,\dots, t_k\}$ of $G$ modulo $F(G)$.
 Let $G_1=\langle g^{t_1}, \dots ,g^{t_k}\rangle$. Clearly, $G_1F(G)/F(G)$ is generated by the conjugacy class of the image of $g$. Since $G/F(G)$ is simple, we have $G_1F(G)=G$. By our assumption, the Cartesian product $T$ of all Sylow $q$-subgroups of $F(G)$ for $q\ne p$ is centralized by all elements $g^{t_i}$. Hence, $[G_1, T]=1$. Let $P$ be the Sylow $p$-subgroup of $F(G)$ (possibly, trivial). Then also $[PG_1, T]=1$, and therefore
 $$\gamma _{\infty }(G)=\gamma _{\infty }(G_1F(G))= \gamma _{\infty }(PG_1).$$
 The image of $\gamma _{\infty }(PG_1)\cap T$ in $G/P$ is contained both in the centre and in the derived subgroup of $PG_1/P$ and therefore is isomorphic to a subgroup of the Schur multiplier of the finite group $G/F(G)$. Since the Schur multiplier of a finite group is finite \cite[Hauptsatz~V.23.5]{hup}, we obtain that $\gamma _{\infty }(G)\cap T=\gamma _{\infty }( PG_1)\cap T$ is finite. Therefore we can assume that $T=1$, in other words, that $F(G)$ is a pro-$p$ group.

 If $|G/F(G)|=p$, then $G$ is a pro-$p$ group, so it is pro\-nil\-po\-tent, which means that $\gamma _{\infty }( G)=1$ and the proof is complete. If $G/F(G)$ is a non-abelian simple group, then we choose another prime $r\ne p$ dividing $|G/F(G)|$ and repeat the same arguments as above with $r$ in place of $p$. As a result, we reduce the proof to the case $F(G)=1$, where the result is obvious.

We now finish the proof of Theorem~\ref{t-eprof} by induction on $|G/F(G)|$. The basis of this induction where $G/F(G)$ is a simple group was proved above. Now suppose that $G/F(G)$ has a nontrivial proper normal subgroup with full inverse image $N$, so that $F(G)<N\lhd G$. Since $F(N)=F(G)$, by induction applied to $N$ the group $\gamma _{\infty }(N)$ is finite. Since $N/\gamma _{\infty }(N)\leq F( G/\gamma _{\infty }(N))$, by induction applied to $G/\gamma _{\infty }(N)$ the group $ \gamma _{\infty }(G/\gamma _{\infty }(N) )$ is also finite. As a result, $\gamma _{\infty }(G) $ is finite, as required.
\end{proof}

\section{Compact groups}
\label{s-comp}
In this section we prove the main Theorem~\ref{t-e}
about compact almost Engel groups. We use the structure theorems for compact groups and the results of the preceding section on profinite almost Engel groups.

Recall that a group $H$ is said to be \emph{divisible} if for every $h\in H$ and every positive integer $k$ there is an element $x\in H$ such that $x^k=h$.

\begin{lemma}\label{l-div}
Suppose that $H$ is divisible group all elements of which are almost right Engel. Then for any $g,x\in H$ there is a positive integer $n(x,g)$ such that $[[g,{}_nx],g]=1$ for all $n\geq n(x,g)$.
\end{lemma}

\begin{proof}
Let $|{\mathscr R}(g)|=m$. Let $h\in H$ be an element such that $h^{m!}=g$. Since $h$ centralizes $g$, by Lemma~\ref{l-c-s} we obtain that $g=h^{m!}$ centralizes ${\mathscr R}(g)$. By the definition of ${\mathscr R}(g)$, there is a positive integer $n(x,g)$ such that $[g,{}_nx]\in {\mathscr R}(g)$ for all $n\geq n(x,g)$. By the above, $[[g,{}_nx],g]=1$ for all $n\geq n(x,g)$.
\end{proof}

\begin{lemma}\label{l-sandw}
If in a finite group $H$ for any $g,x\in H$ there is a positive integer $n(x,g)$ such that $[[g,{}_nx],g]=1$ for all $n\geq n(x,g)$, then $H$ is nilpotent.
\end{lemma}

\begin{proof}
If $H$ is not nilpotent, then by a theorem of Frobenius \cite[10.3.2]{rob} there is a prime $p$ and a $p$-subgroup $P$ such that $N_H(P)/C_H(P)$ is not a $p$-group. Then there is a $p'$-element $g$ normalizing but not centralizing $P$. For the abelian group $V=P/P'$ with the induced action of $g$ we obtain $1\ne [V,g]=\{[x,{}_ng]\mid x\in [V,g]\}$ for any $n\in \N$, so that $[x,{}_ng]\ne 1$ if $1\ne x\in [V,g]$. On the other hand, for $n\geq n(gx,g)$ we have
 $$1=[[g,{}_ngx],g]= [[[x,g]^{-1},{}_{n-1}gx],g] =[x,{}_{n+1}g]^{-1},$$
 a contradiction.
\end{proof}

We are ready to prove Theorem~\ref{t-e}.

\begin{proof}[Proof of Theorem~\ref{t-e}] Let $G$ be a compact group all elements of which are almost right Engel; we need to show that there is a finite subgroup $N$ such that $G/N$ is locally nilpotent.
 By the well-known structure theorems (see, for example, \cite[Theorems~9.24 and 9.35]{h-m}), the connected component of the identity $G_0$ in $G$ is a divisible group such that $G_0/Z(G_0)$ is a Cartesian product of simple compact Lie groups, while the quotient $G/G_0$ is a profinite group. Note that simple compact Lie groups are linear groups.

 \begin{lemma}\label{l-abelian}
 The connected component of the identity $G_0$ in $G$ is an abelian group.
 \end{lemma}

 \bp
 By Lemma~\ref{l-div}, for any $x,g\in G_0$ there is $n(x,g)\in \N$ such that \begin{equation}\label{e-sandw}
 [[g,{}_nx],g]=1
 \end{equation}
 for all $n\geq n(x,g)$. This property is obviously inherited by any section of $G_0$.

 We need to show that $G_0=Z(G_0)$. Suppose the opposite, and let $S$ be a nontrivial simple compact Lie group that is a subgroup of $G_0/Z(G_0)$. Let $F$ be any finitely generated subgroup of $S$. Being a linear group, $F$ is residually finite by Mal'cev's theorem \cite{mal}. In view of relations \eqref{e-sandw}, then $F$ is residually nilpotent by Lemma~\ref{l-sandw}. For any $f\in F$ the right Engel sink ${\mathscr R}_F(f)$ in $F$ is finite. Hence there is a normal subgroup $N$ of $F$ such that $F/N$ is nilpotent and ${\mathscr R}_F(f)\cap N=\{1\}$. Therefore ${\mathscr R}_F(f)=\{1\}$ for any $f\in F$, whence $F$ is an Engel group. Linear Engel groups are locally nilpotent by the Garashchuk--Suprunenko theorem \cite{gar-sup} (see also \cite{grb}), so $F$ is nilpotent. Hence, $S$ is locally nilpotent, a contradiction.
\ep

We proceed with the proof of Theorem~\ref{t-e}.

\begin{lemma}\label{l-eng}
For every $g\in G$ and for any $x\in G_0$ we have
$[x,{}_ng]=1$ for large enough~$n$.
\end{lemma}

\begin{proof}
Since $G_0$ is abelian by Lemma~\ref{l-abelian}, the group $H=G_0\langle g\rangle$ is metabelian; therefore $g$ has a finite left Engel sink and ${\mathscr E}_H(g)\subseteq {\mathscr R}_H(g^{-1})$ by Lemma~\ref{l-metab}. The conclusion of the lemma is equivalent to ${\mathscr E}_H(g)=\{1\}$.

Suppose the opposite and choose $1\ne z\in {\mathscr E}_H(g)$. By Lemma~\ref{l-sink} then $z=[z,{}_ng]$ for some $n\geq 1$.
By Lemma~\ref{l-metab2}(a), ${\mathscr E}_H(g)$ is a subgroup contained in $G_0$. Hence we can choose $z_1$ in ${\mathscr E}(g)$ of some prime order $p$. Again by Lemma~\ref{l-sink} we have $z_1=[z_1,{}_mg]$ with $m\geq 1$. In the divisible group $G_0$ for every $k=1,2,\dots $ there is an element $z_k$ such that $z_k^{p^k}=z_1$. We have $y_k= [z_k, {}_{n_k}g]\in {\mathscr E}_H(g)$ for some $n_k\geq 1$.
 Then $y_k^{p^k}= [z_k^{p^k}, {}_{n_k}g]=[z_1,{}_{n_k}g]$, which is an element of the orbit of $z_1$ in ${\mathscr E}_H(g)$ under the mapping $x\to [x,g]$ and therefore has the same order $p=|z_1|$ by Lemma~\ref{l-metab2}(b). Thus, $y_k$ is an element of ${\mathscr E}_H(g)$ of order exactly $p^{k+1}$, for $k=1,2,\dots $. As a result, ${\mathscr E}_H(g)$ is infinite, a contradiction.
\end{proof}

Applying Theorem~\ref{t-eprof} to the profinite group $\bar G=G/G_0$ all elements of which are almost right Engel, we obtain a finite normal subgroup $D$ with locally nilpotent quotient. Then all elements of $\bar G$ are almost left Engel, and $D$ contains all left Engel sinks ${\mathscr E}_{\bar G}(g)$ of elements $g\in\bar G$. Hence $D$ also contains the subgroup $E$ generated by them:
$$
E:=\langle {\mathscr E}_{\bar G}(g)\mid g\in \bar G\rangle\leq D.
$$
Clearly, $E$ is a normal finite subgroup of $\bar G$. Note
that $\bar G/E$ is also locally nilpotent by the Wilson--Zelmanov theorem \cite[Theorem~5]{wi-ze}, since this is an Engel profinite group.

We now consider the action of $\bar G$ by automorphisms on $G_0$ induced by conjugation.

\begin{lemma}\label{l-central}
The subgroup $E$ acts trivially on $G_0$.
\end{lemma}

\begin{proof}
The abelian divisible group $G_0$ is a direct product $A_0\times\prod _pA_p$ of a torsion-free divisible group $A_0$ and Sylow subgroups $A_p$ over various primes $p$. Clearly, every Sylow subgroup $A_p$ is normal in $G$.

First we show that $E$ acts trivially on each $A_p$. It is sufficient to show that for every $g\in \bar G$ every element $z\in {\mathscr E}_{\bar G}(g)$ acts trivially on $A_p$. Consider the action of $\langle z, g\rangle$ on $A_p$. Note that $\langle z, g\rangle=\langle z^{\langle g\rangle}\rangle\langle g\rangle$, where $\langle z^{\langle g\rangle}\rangle$ is a finite $g$-invariant subgroup, since it is contained in the finite subgroup $E$. For any $a\in A_p$ the subgroup
$$
\langle a^{\langle g\rangle}\rangle=\langle a,[a,g], [a,g,g],\dots \rangle
$$
is a finite $p$-group by Lemma~\ref{l-eng}, and this subgroup is $g$-invariant.
 Its images under the action of elements of the finite group $\langle z^{\langle g\rangle}\rangle$ generate a finite $p$-group $B$, which is $\langle z, g\rangle$-invariant. Lemma~\ref{l-eng} implies that the image of $\langle z, g\rangle$ in its action on $B$ must be a $p$-group. Indeed, otherwise this image contains a $p'$ element $y$ that acts non-trivially on the Frattini quotient $V=B/\Phi (B)$. Then $V=[V,y]$ and $C_V(y)=1$, whence $[V,y]=\{[v,g]\mid v\in [V,y]\}$ and therefore also $[V,y]=\{[v,{}_ny]\mid v\in [V,y]\} $ for any $n$, contrary to Lemma~\ref{l-eng}. But since $z$ is an element of ${\mathscr E}_{\bar G}(g)$, by Lemma~\ref{l-sink} we have $z=[z,g,\dots ,g]$ with at least one occurrence of $g$. Since a finite $p$-group is nilpotent, this implies that the image of $z$ in its action on $B$ must be trivial. In particular, $z$ centralizes $a$. As a result $E$ acts trivially on $A_p$, for every prime $p$.

 We now show that $E$ also acts trivially on the quotient $W=G_0/\prod _pA_p$ of $G_0$ by its torsion part. Note that $W$ can be regarded as a vector space over ${\mathbb Q}$. Every element $g\in E$ has finite order and therefore by Maschke's theorem $W=[W,g]\times C_W(g)$. If $[W,g]\ne 1$, then $[W,g]=\{[w,{g,\dots ,g}]\mid w\in [W,g]\} $ with $g$ repeated $n$ times, for any $n$. This contradicts Lemma~\ref{l-eng}.

 Thus, $E$ acts trivially both on $W$ and on $\prod _pA_p$. Then any automorphism of $G_0$ induced by conjugation by $g\in E$ acts on every element $a\in A_0$ as $a^g=at$, where $t=t(a)$ is an element of finite order in $G_0$. Since $a^{g^i}=at^i$, the order of $t$ must divide the order of $g$. Assuming the action of $E$ on $G_0$ to be non-trivial, choose an element $g\in E$ acting on $G_0$ as an automorphism of some prime order $p$. Then there is $a\in A_0$ such that $a^g=at$, where $t\in G_0$ has order $p$. For any $k=1,2,\dots $ there is an element $a_k\in A_0$ such that $a_k^{p^k}=a$. Then $a_k^g=a_kt_k$, where $t_k^{p^k}=t$. Thus $|t_k|=p^{k+1}$, and therefore $p^{k+1}$ divides the order of $g$, for every $k=1,2,\dots$. We arrived at a contradiction since $g$ has finite order.
\end{proof}

\begin{lemma}\label{l-left}
 Every element of $G$ has finite left Engel sink.
\end{lemma}

\bp
Let $F$ be the full inverse image of $E$ in $G$. Given an element $g\in G$, for any $x\in G$ there is $m=m(x,g)$ such that $y=[x,{}_{m}g]\in F$, since $G/F$ is locally nilpotent. Recalling formula \eqref{r-l} we have
$$
[x,{}_{m+n+1}g]=[y,{}_{n+1}g]=[g^{-1},{}_ng^{y^{-1}}]^{yg}.
$$
For large enough $n$, the right-hand side belongs to ${\mathscr R}(g^{-1})^{yg}$, and therefore to the union $\bigcup _{f\in F}{\mathscr R}(g^{-1})^{fg}$ (which is independent of $x$). By Lemma~\ref{l-central}, $|F/Z(F)|$ is finite and therefore there are only finitely many different conjugates ${\mathscr R}(g^{-1})^{f}$ of the finite subset ${\mathscr R}(g^{-1})\subseteq F$. Therefore the union $\bigcup _{f\in F}{\mathscr R}(g^{-1})^{fg}$ is a finite left Engel sink for $g$.
\ep

Thus, all elements of the compact group $G$ have finite left Engel sinks, and therefore by the main result of \cite{khu-shu162} the group $G$ has a finite normal subgroup with locally nilpotent quotient. The proof of Theorem~\ref{t-e} is complete.
 \end{proof}

\begin{corollary}
\label{c-em}
Let $G$ be a compact group such that for some positive integer $m$ every element has a finite right Engel sink ${\mathscr R}(g)$ of cardinality at most $m$. Then $G$ has a finite normal subgroup $N$ of order bounded in terms of $m$ such that $G/N$ is locally nilpotent.
\end{corollary}

\begin{proof}
By Theorem~\ref{t-e} the group $G$ is finite-by-(locally nilpotent). Therefore every abstract finitely generated subgroup $H$ of $G$ is finite-by-nilpotent and residually finite. By Theorem~\ref{t-finite}, every finite quotient of $H$ has nilpotent residual of $m$-bounded order. Hence $\gamma _{\infty}(H)$ is also finite of $m$-bounded order $\leq f(m)$, and $H/\gamma _{\infty}(H)$ is nilpotent. Standard properties of nilpotent groups ensure that $H$ has only finitely many normal subgroups of order $\leq f(m)$ with nilpotent quotient. Then the standard inverse limit argument can be applied, by which the group $G$ has a normal subgroup of $m$-bounded order $\leq f(m)$ with locally nilpotent quotient.
\end{proof}

\section*{Acknowledgements}
The authors thank Gunnar Traustason and John Wilson for stimulating discussions.

The first author was supported by the Russian Science Foundation, project no. 14-21-00065,
and the second by FAPDF, Brazil. The first author thanks CNPq-Brazil and the University of Brasilia for support and hospitality that he enjoyed during his visits to Brasilia.

\end{document}